\documentclass[12pt,a4paper]{article}
\usepackage{latexsym,amssymb,amsfonts,amsmath,amsthm,nccmath,float,enumitem,setspace,bm}
\usepackage[hidelinks]{hyperref}
\usepackage[dvipsnames]{xcolor}
\usepackage[margin=2cm]{geometry}
\usepackage{tikz}
\usetikzlibrary{arrows}
\usetikzlibrary{calc}
\usepackage{color}
\hypersetup{
    colorlinks,
    linkcolor={red!80!black},
    citecolor={blue!80!black},
    urlcolor={blue!80!black}
}

\newtheorem{theorem}{Theorem}
\newtheorem{corollary}[theorem]{Corollary}
\newtheorem{lemma}[theorem]{Lemma}

\theoremstyle{definition}

\newtheorem{example}[theorem]{Example}
\def \deg {{\rm deg}}

\def \leq {\leqslant}
\def \geq {\geqslant}
\def \Z {\mathbb{Z}}

\def \P {\mathcal{P}}
\def \G {\mathcal{G}}
\def \D {\mathcal{D}}
\def \C {\mathcal{C}}
\def \S {\mathcal{S}}
\def \Pfin {\mathcal{P}_{\rm fin}}

\let\oldproofname=\proofname
\renewcommand{\proofname}{\rm\bf{\oldproofname}}

\date{}

\begin{document}

\title{Generating infinite digraphs by derangements}

\author{Daniel Horsley\thanks{School of Mathematics, Monash University, Wellington Rd, Clayton VIC 3800, Australia ({\tt danhorsley@gmail.com})},\ \ Moharram Iradmusa\thanks{Faculty of Mathematical Sciences, Shahid Beheshti University, Tehran, Iran ({\tt m\_iradmusa@sbu.ac.ir})},\ \  Cheryl~E.~Praeger\thanks{ Centre for the Mathematics of Symmetry and Computation, The University of Western Australia, 35
Stirling Highway, Crawley WA 6009, Australia ({\tt  cheryl.praeger@uwa.edu.au})}}
\maketitle

\begin{abstract}
A set $\S$ of derangements (fixed-point-free permutations) of a set $V$ generates a digraph with vertex set $V$ and arcs $(x,x^\sigma)$ for $x\in V$ and $\sigma\in\S$. We address the problem of characterising those infinite (simple loopless) digraphs which are generated by finite sets of derangements. The case of finite digraphs was addressed in earlier work by the second and third authors.  A criterion is given for derangement generation which resembles the criterion given by De Bruijn and Erd\H{o}s for vertex colourings of graphs in that the property for an infinite digraph is determined by properties of its finite sub-digraphs.
The derangement generation property for a digraph is linked with the existence of a finite $1$-factor cover for an associated bipartite (undirected) graph.
\end{abstract}

\setstretch{1.1}

\section{Introduction}

A \emph{derangement} of a set $V$ is a bijection $\sigma: V \rightarrow V$ such that $x^{\sigma} \neq x$ for all $x \in V$. A \emph{digraph} $D=(V,E)$ consists of a (possibly infinite) vertex set $V=V(D)$ and an arc set  $E=E(D) \subseteq \{(x,y): x,y \in V, x \neq y\}$. Thus our digraphs are simple and loopless by definition. Each set $\S$ of derangements of $V$ corresponds to a digraph, namely
$(V, E(\S))$ where $E(\S):=\{(x,x^\sigma): x \in V, \sigma \in \S\}$. Moreover we say that a set $\S$ of derangements of $V$ \emph{generates} a digraph $D=(V,E)$ if $E=E(\S)$. Digraphs of this kind  were introduced by the second and third author in \cite{IP}, where they were called \emph{derangement action digraphs} and denoted $\overrightarrow{DA}(V,\S) = (V, E(\S))$. Special attention was paid in \cite{IP} to those where the arc set $E$ is symmetric in the sense that $E$ is equal to $E^*:=\{(y,x) : (x,y)\in E\}$, and in this case  the digraph is viewed as a simple undirected graph. It was proved in \cite[Theorem 1.6]{IP} that the class of derangement action digraphs includes all finite regular simple digraphs.
It was also shown in \cite[Theorem 1.10]{IP} that any finite regular simple graph which either has even valency, or has odd valency but contains a perfect matching, or is bipartite, or is vertex-transitive, can be generated by a set of derangements that is closed and self-inverse (see \cite{IP} for the relevant definitions). In \cite{IP} it is always assumed that the generating set $\S$ of derangements is finite. This requirement is always satisfied in the case of finite digraphs, but becomes important when infinite digraphs are considered. Accordingly the authors asked which infinite regular simple graphs and digraphs are generated by finitely many derangements,~\cite[Problem 2]{IP}.  Here we characterise such digraphs.

We introduce some standard terminology. Let $D=(V,E)$ be a digraph. For a vertex $x \in V$, the \emph{out-neighbourhood of $x$} is $N_D^+(x) :=\{y \in V:(x,y) \in E\}$ and the \emph{in-neighbourhood of $x$} is $N_D^-(x) :=\{y \in V:(y,x) \in E\}$. The \emph{out-degree} of $x$ is  $\deg^+_D(x)=|N_D^+(x)|$ and the \emph{in-degree} of $x$ is  $\deg^-_D(x)=|N_D^-(x)|$. The digraph $D$ is said to be \emph{$k$-regular} if all in-degrees and all out-degrees are equal to $k$. For a subset $S$ of $V$ we define $N_D^+(S):=\bigcup_{x \in S}N_D^+(x)$ and $N_D^-(S):=\bigcup_{x \in S}N_D^-(x)$. Our main result characterises, for each positive integer $k$, the digraphs that can be generated by at most $k$ derangements.

\begin{theorem}\label{T:digraphkCharacterisation}
Let $D$ be a (possibly infinite) digraph and let $k$ be a positive integer. Then $D$ can be generated by at most $k$ derangements if and only if
\begin{itemize}
    \item[\textup{(i)}]
$\deg^+_D(x) \leq k$ and $\deg^-_D(x) \leq k$ for each $x \in V(D)$; and
    \item[\textup{(ii)}]
$\displaystyle{k\bigl(|N^+_D(T)|-|T|\bigr) \geq \medop\sum_{x \in N^+_D(T)} \deg^-_D(x) - \medop\sum_{x \in T}\deg_D^+(x)}$ for each finite subset $T$ of $V(D)$; and
    \item[\textup{(iii)}]
$\displaystyle{k\bigl(|N^-_D(T)|-|T|\bigr) \geq \medop\sum_{x \in N^-_D(T)} \deg^+_D(x) - \medop\sum_{x \in T}\deg^-_D(x)}$ for each finite subset $T$ of $V(D)$.
\end{itemize}
\end{theorem}

The conditions (i), (ii) and (iii) can be seen to be necessary for $D$ to be generated by at most $k$ derangements as follows. Suppose that $D$ is a digraph generated by at most $k$ derangements. Because a derangement in the generating set of $D$ induces exactly one arc $(x,x^\sigma)$ \emph{from each vertex $x$}, and exactly one arc $(x^{\sigma^{-1}},x)$ \emph{to each vertex $x$}, condition (i) must hold. Let $T$ be a finite subset of $V$. Any derangement in the generating set of $D$ must induce $|T|$ arcs from vertices in $T$ to vertices in $N^+_D(T)$ and hence $|N^+_D(T)|-|T|$ arcs from vertices outside of $T$ to vertices in $N^+_D(T)$. However, there are $\sum_{x \in N^+_D(T)} \deg^-_D(x) - \sum_{x \in T}\deg_D^+(x)$ arcs from vertices outside of $T$ to vertices in $N^+_D(T)$, and hence condition (ii) must hold. A similar argument with the direction of the arcs reversed shows that condition (iii) must also hold.

We prove Theorem~\ref{T:digraphkCharacterisation} by representing digraphs as bipartite graphs and proving an analogous result concerning (possibly infinite) bipartite graphs. A graph $G=(V,E)$ consists of a (possibly infinite) vertex set $V=V(G)$ and an edge set  $E=E(G) \subseteq \{\{x,y\}: x,y \in V, x \neq y\}$. For a vertex $x \in V$, we define the \emph{neighbourhood of $x$} to be $N_G(x)=\{y \in V:\{x,y\} \in E\}$ and the \emph{degree} of $x$ to be $\deg_G(x)=|N_G(x)|$. The graph $G$ is said to be \emph{$k$-regular} if all vertices have degree $k$. For a subset $S$ of $V$ we define $N_G(S):=\bigcup_{x \in S}N_G(x)$.  A \emph{$1$-factor} $F$ of $G$ is a set of edges of $G$ such that each vertex of $G$ is incident with exactly one edge in $F$. A \emph{$1$-factor cover} of $G$ is a set $\mathcal{F}$ of $1$-factors of $G$ such that each edge of $G$ is in at least one $1$-factor in $\mathcal{F}$.

\begin{theorem}\label{T:graphkCharacterisation}
Let $G$ be a (possibly infinite) bipartite graph with bipartition $\{V_1,V_2\}$ and let $k$ be a positive integer. Then $G$ has a $1$-factor cover with at most $k$ $1$-factors if and only if
\begin{itemize}
    \item[\textup{(i)}]
$\deg_G(x) \leq k$ for each $x \in V(G)$; and
    \item[\textup{(ii)}]
$\displaystyle{k\bigl(|N_G(T)|-|T|\bigr) \geq \medop\sum_{x \in N_G(T)} \deg_G(x) - \medop\sum_{x \in T}\deg_G(x)}$ for each finite subset $T$ of $V_1$ or $V_2$.
\end{itemize}
\end{theorem}

Similar arguments to those given above establish the necessity of these conditions for the existence of a $1$-factor cover of $G$ with at most $k$ $1$-factors. Bonisoli and Cariolaro introduced the study of $1$-factor covers in \cite{BonCar}, and Cariolaro and
Rizzi~\cite{CarRiz} gave a characterisation of the family of finite bipartite graphs which have a $1$-factor cover with at most $k$ $1$-factors. Theorems~\ref{T:digraphkCharacterisation} and \ref{T:graphkCharacterisation} resemble the De Bruijn-Erd\H{o}s theorem \cite{deBErd} in that they characterise a property of an infinite graph in terms of properties of its finite subgraphs. The proofs of Theorems~\ref{T:digraphkCharacterisation} and \ref{T:graphkCharacterisation} depend on the axiom of choice, in the guise of Zorn's lemma (see Lemma~\ref{L:finiteCondition}).

We state one consequence of Theorems~\ref{T:digraphkCharacterisation} and~\ref{T:graphkCharacterisation} that generalises \cite[Theorem 1.6(b)]{IP} and \cite[Theorem 1.10(c)]{IP} respectively, and gives some explicit families of infinite digraphs and graphs that can be generated by derangements, \cite[Problem 2]{IP}.

\begin{corollary}\label{C:regDigraph}
Let $k$ be a positive integer.
\begin{enumerate}
\item[(a)] A $k$-regular digraph can be generated by $k$ derangements but no fewer.
\item[(b)] A $k$-regular graph can be generated by $k$ derangements but no fewer.
\end{enumerate}
\end{corollary}

The rest of this paper is organised as follows. In the next section we describe the notion of the bipartite double of a digraph and show that Theorem~\ref{T:graphkCharacterisation} implies Theorem~\ref{T:digraphkCharacterisation}, and that these two results imply Corollary~\ref{C:regDigraph}. In Section~\ref{S:infiniteGeneration} we characterise those digraphs with finite in- and out-degrees that can be generated by finitely many derangements. This characterisation follows without too much effort from well known results. In Section~\ref{S:finiteGeneration} we undertake the more substantial task of proving Theorem~\ref{T:graphkCharacterisation}. We conclude with Section~\ref{S:examples} in which we exhibit examples of digraphs with low maximum degree that require many (including infinitely many) derangements to generate them.

\section{Bipartite doubles: deriving Theorem~\ref{T:digraphkCharacterisation} and Corollary~\ref{C:regDigraph} from  Theorem~\ref{T:graphkCharacterisation}}\label{S:doubles}

The notion of a bipartite double of a digraph provides the link between Theorems~\ref{T:digraphkCharacterisation} and \ref{T:graphkCharacterisation}. The \emph{bipartite double} of a digraph $D$ is the bipartite graph $G$ with parts $V(D) \times \{1\}$ and $V(D) \times \{2\}$ and edge set $\{\{(x,1),(y,2)\}:(x,y) \in E(D)\}$.

Let $D$ be a digraph generated by some set $\S$ of derangements of a set $V$. The arc subset generated by a single derangement $\sigma \in \S$ is $E(\sigma):= \{(x,x^\sigma): x \in V\}$ and by the definition of a derangement, for each $x\in V$, $E(\sigma)$ contains exactly
one arc of the form $(x,y)$ and one arc of the form $(y',x)$ (for some $y, y'\in V\setminus\{x\}$). Thus $E(\sigma)$ comprises exactly one arc of $D$ into each vertex, and also comprises exactly one arc of $D$ out of each vertex. The edge subset of the bipartite double $G$ of $D$ corresponding to $E(\sigma)$ is $F(\sigma):=\{ \{ (x,1),(x^\sigma,2) \} : x \in V\}$ and each vertex of $G$ is incident with exactly one  edge
of $F(\sigma)$. Thus $F(\sigma)$ is a $1$-factor of $G$. Since each arc of $D$ is generated by some derangement in $\S$, the $1$-factors $F(\sigma)$, for $\sigma\in \S$, form a $1$-factor cover of $G$ with $|\S|$ $1$-factors. Conversely, every $1$-factor cover $\mathcal{F}$ of $G$ corresponds to a set of $|\mathcal{F}|$ derangements that generates $D$.

With this equivalence established it is easy to see that Theorem~\ref{T:graphkCharacterisation} implies Theorem~\ref{T:digraphkCharacterisation}.

\begin{proof}[\textbf{\textup{Proof that Theorem~\ref{T:graphkCharacterisation} implies Theorem~\ref{T:digraphkCharacterisation}.}}]
Let $D$ be a digraph and let $k$ be a positive integer. Let $G$ be the bipartite double of $D$ and recall that $G$ has parts $V(D) \times \{1\}$ and $V(D) \times \{2\}$. By our discussion immediately above, $D$ can be generated by at most $k$ derangements if and only if $G$ has a $1$-factor cover with at most $k$ $1$-factors. Assuming Theorem~\ref{T:graphkCharacterisation} holds, $G$ has a $1$-factor cover with at most $k$ $1$-factors if and only if $G$ satisfies (i) and (ii) of Theorem~\ref{T:graphkCharacterisation}. The rest of the proof will now follow from the fact that $G$ is the bipartite double of $D$. Observe that $G$ satisfies condition (i) of Theorem~\ref{T:graphkCharacterisation} if and only if $D$ satisfies condition (i) of Theorem~\ref{T:digraphkCharacterisation}. Further $G$ satisfies condition (ii) of Theorem~\ref{T:graphkCharacterisation}, for each finite subset $T$ of $V(D) \times \{1\}$, if and only if $D$ satisfies condition (ii) of Theorem~\ref{T:digraphkCharacterisation}; and $G$ satisfies condition (ii) of Theorem~\ref{T:graphkCharacterisation}, for each finite subset $T$ of $V(D) \times \{2\}$, if and only if $D$ satisfies condition (iii) of Theorem~\ref{T:digraphkCharacterisation}.
\end{proof}

We also show that Theorem~\ref{T:digraphkCharacterisation} implies Corollary~\ref{C:regDigraph}.

\begin{proof}[\textbf{\textup{Proof that Theorem~\ref{T:digraphkCharacterisation} implies Corollary~\ref{C:regDigraph}.}}]
Let $D$ be a $k$-regular digraph. By Theorem~\ref{T:digraphkCharacterisation}(i), $D$ cannot be generated by fewer than $k$ derangements. Let $T$ be a finite subset of $V(D)$. Then $\sum_{x \in N^+_D(T)} \deg^-_D(x)=k|N^+_D(T)|$ and $\sum_{x \in T}\deg_D^+(x)=k|T|$ and consequently condition (ii) of Theorem~\ref{T:digraphkCharacterisation} holds. Similarly $\sum_{x \in N^-_D(T)} \deg^+_D(x)=k|N^-_D(T)|$ and $\sum_{x \in T}\deg_D^-(x)=k|T|$ and consequently condition (iii) of Theorem~\ref{T:digraphkCharacterisation} holds. So, by Theorem~\ref{T:digraphkCharacterisation}, $D$ can be generated by $k$ derangements, proving Corollary~\ref{C:regDigraph}(a).  To prove part (b), it is enough to replace each edge of the graph with two arcs with opposite directions. Now we have a $k$-regular digraph and so part (a) of Corollary~\ref{C:regDigraph} implies part (b).
\end{proof}

\section{Possibly infinite sets of derangements}\label{S:infiniteGeneration}

A digraph $D$ is \emph{locally finite} if $\deg_D^+(x)$ and $\deg_D^-(x)$ are finite for each vertex $x$ of $D$. In this section we characterise the locally finite digraphs which can be generated by some (possibly infinite) set of derangements. The characterisation follows without too much difficulty from well known results.

By considering bipartite doubles, it suffices to characterise those locally finite bipartite graphs that have a $1$-factor cover. Graphs $G$ such that each edge of $G$ is in some $1$-factor of $G$ are commonly called \emph{$1$-extendable}. Any graph with a $1$-factor cover is clearly $1$-extendable. Conversely, for any $1$-extendable graph $G$ and each edge $\{u,v\}$ of $G$, let  $F_{\{u,v\}}$ denote a $1$-factor of $G$ containing $\{u,v\}$. Then  $\{F_{\{u,v\}}:\{u,v\} \in E(G)\}$ is a $1$-factor cover of $G$. Thus it suffices,  in fact, to characterise those locally finite bipartite graphs that are $1$-extendable. For a set $X$, let $\P(X)$ denote the set of all subsets of $X$. We also denote by $\Pfin(X)$ the set of all finite subsets of $X$.

We will make use of the following result of Rado \cite[Theorem II]{Rad} which extends a famous theorem of Hall \cite{Hal} concerning finite bipartite graphs to the case of locally finite graphs (note that Rado did not bother to mention the obvious `only if' direction in his statement).

\begin{theorem}[\cite{Rad}]\label{T:infiniteHall}
A locally finite bipartite graph $G$ with bipartition $\{V_1,V_2\}$ has a $1$-factor if and only if
$|N_G(T)| \geq |T|$ for each $T \in \Pfin(V_1) \cup \Pfin(V_2)$.
\end{theorem}

In Lemma~\ref{T:infinite1Extendible} below we use Theorem~\ref{T:infiniteHall} to obtain a criterion for 1-extendability of locally finite bipartite graphs.
A version of Lemma~\ref{T:infinite1Extendible}  for finite graphs was proved in \cite{Het} (see also \cite[Theorem 4.1.1]{LovPlu}).

\begin{lemma}\label{T:infinite1Extendible}
A locally finite bipartite graph $G$ with bipartition $\{V_1,V_2\}$ is $1$-extendable if and only if, for each $T \in \Pfin(V_1) \cup \Pfin(V_2)$,
\begin{itemize}[nosep]
    \item[\textup{(i)}]
$|N_G(T)| \geq |T|$; and
    \item[\textup{(ii)}]
if $|N_G(T)|=|T|$, then $N_G(N_G(T))=T$.
\end{itemize}
\end{lemma}

\begin{proof}
If there is a set $T \in \Pfin(V_1) \cup \Pfin(V_2)$ for which (i) fails, then by Theorem~\ref{T:infiniteHall}, $G$ has no $1$-factor at all. If there is a set $T \in \Pfin(V_1) \cup \Pfin(V_2)$ for which (ii) fails, then $|N_G(T)|=|T|$ and there is an edge $\{x, y\}$ in $G$ such that $x \notin T$ and $y \in N_G(T)$. Let $G'$ be the graph obtained from $G$ by deleting the vertices $x$ and $y$ and all of the edges incident with them. Then $|N_{G'}(T)|=|T|-1$ and $G'$ has no perfect matching by Theorem~\ref{T:infiniteHall}. So $G$ has no $1$-factor containing the edge $\{x, y\}$.

Conversely, suppose that there is an edge $\{x, y\}$ of $G$ that lies in no $1$-factor of $G$. We will prove that (i) or (ii) fails for some finite subset of $V_1$ or $V_2$.  Let $G'$ be the graph obtained from $G$ by deleting the vertices $x$ and $y$ and all of the edges incident with them. Note that $G'$ has no 1-factor since $\{x, y\}$ lies in no $1$-factor of $G$. For each $i \in \{1,2\}$, let $V'_i=V_i \setminus \{x,y\}$, noting that $|V_i \cap \{x,y\}|=1$. By Theorem~\ref{T:infiniteHall} there is a set $T \in \Pfin(V'_1) \cup \Pfin(V'_2)$ such that $|N_{G'}(T)|<|T|$.  Without loss of generality we may assume that $T\subseteq V_1'$ and that $x\in V_1$ so that $y\in V_2$ and, by the definition of $G'$, $N_G(T)\setminus N_{G'}(T) \subseteq \{y\}$. If in fact $|N_G(T)|=|N_{G'}(T)|$ holds, then $|N_G(T)|<|T|$ and so (i) fails. Thus we may assume that $|N_G(T)|\neq|N_{G'}(T)|$, and hence $N_G(T)=N_{G'}(T)\cup\{y\}$ because $N_G(T)\setminus N_{G'}(T) \subseteq \{y\}$. Thus $x \in N_G(N_G(T)) \setminus T$. Further, since $|N_{G'}(T)|<|T|$, we have either $|N_{G'}(T)|<|T|-1$ or $|N_{G'}(T)|=|T|-1$. In the former case, $|N_{G}(T)|= |N_{G'}(T)|+1<|T|$ and (i) fails, while in the latter case $|N_G(T)|=|T|$ and (ii) fails, since $x \in N_G(N_G(T)) \setminus T$.
\end{proof}

By applying Lemma~\ref{T:infinite1Extendible} to bipartite doubles, we can characterise the locally finite digraphs which can be generated by some (possibly infinite) set of derangements.

\begin{theorem}\label{T:infinite1ExtendibleDigraph}
A locally finite digraph $D$ can be generated by some (possibly infinite) set of derangements if and only if, for all finite $T \subseteq V(D)$
the following three conditions all hold.
\begin{itemize}[nosep]
    \item[\textup{(i)}]
$|N^+_{D}(T)| \geq |T|$ and $|N^-_{D}(T)| \geq |T|$;
    \item[\textup{(ii)}]
if $|N^+_{D}(T)|=|T|$, then $N^-_{D}(N^+_{D}(T))=T$;
    \item[\textup{(iii)}]
if $|N^-_{D}(T)|=|T|$, then $N^+_{D}(N^-_{D}(T))=T$.
\end{itemize}
\end{theorem}

\begin{proof}
The proof proceeds along similar lines to our proof that Theorem~\ref{T:graphkCharacterisation} implies Theorem~\ref{T:digraphkCharacterisation}. Let $D$ be a locally finite digraph, let $G$ be the bipartite double of $D$ and recall that $G$ has parts $V(D) \times \{1\}$ and $V(D) \times \{2\}$.
As discussed above, $D$ can be generated by a set of derangements if and only if $G$ has a $1$-factor cover or, equivalently, if and only if $G$ is $1$-extendable. By Lemma~\ref{T:infinite1Extendible}, $G$ is $1$-extendable if and only if conditions (i) and (ii) of Lemma~\ref{T:infinite1Extendible} hold for all finite subsets of $V(D) \times \{1\}$ and $V(D) \times \{2\}$. Let $T$ be a finite subset of $V(D)$ and observe the following.
\begin{itemize}[nosep,leftmargin=*]
    \item
(i) of this theorem holds for $T$ if and only if (i) of Lemma~\ref{T:infinite1Extendible} holds for $T \times \{1\}$ and $T \times \{2\}$.
    \item
(ii) of this theorem holds for $T$ if and only if (ii) of Lemma~\ref{T:infinite1Extendible} holds for $T \times \{1\}$.
    \item
(iii) of this theorem holds for $T$ if and only if (ii) of Lemma~\ref{T:infinite1Extendible} holds for $T \times \{2\}$.
\end{itemize}
Thus (i)-(iii) of this theorem hold for all finite subsets of $V(D)$ if and only if (i) and (ii) of Lemma~\ref{T:infinite1Extendible} hold for all finite subsets of $V(D) \times \{1\}$ and $V(D) \times \{2\}$, and in turn this  occurs if and only if $D$ can be generated by some set of derangements.
\end{proof}

\section{Finite sets of derangements}\label{S:finiteGeneration}

In this section we prove Theorem~\ref{T:graphkCharacterisation} and hence, by the results in Section~\ref{S:doubles}, also Theorem~\ref{T:digraphkCharacterisation} and Corollary~\ref{C:regDigraph}. Although Theorem~\ref{T:graphkCharacterisation} concerns (simple) graphs, our proof of it, based on the approach of Cariolaro and Rizzi \cite{CarRiz}, will rely heavily on multigraphs. We now introduce some notation concerning them.

\subsection{Multigraph definitions and notation}

A \emph{multigraph}  $G=(V,E)$ has vertex set $V=V(G)$ and edge set $E=E(G)$ such that each edge is incident with exactly two distinct vertices (that is, $G$ has no loops). Two vertices incident with some edge are called \emph{adjacent}.  For $x\in V(G)$, we denote by  $N_G(x)$ the set of vertices adjacent to $x$, and for a subset $T$ of $V(G)$,  let $N_G(T)=\bigcup_{x\in T}N_G(x)$.  For distinct vertices $x, y\in V(G)$, let $\mu_G(x,y)$ denote the number of edges between (incident with) $x$ and $y$. The \emph{degree} $\deg_G(x)$ of a vertex $x$ in a multigraph $G$ is the number of edges incident with it, so $\deg_G(x) =\sum_{y\in N_G(x)}\mu_G(x,y) \geq |N_G(x)|$ but equality need not hold. A multigraph is \emph{$k$-regular} if each of its vertices has degree $k$. A multigraph $G$ is a (simple) \emph{graph} if $\mu_G(x,y)\leq 1$ for all $x, y\in V(G)$. (Note that in this case we can identify each edge with the pair of vertices it is incident with and so recover the definition of graph given in the introduction.)

A multigraph $G_1$ is said to be a \emph{subgraph} of a multigraph $G_2$ if $V(G_1) \subseteq V(G_2)$ and $\mu_{G_1}(x,y) \leq \mu_{G_2}(x,y)$ for all distinct $x,y \in V(G_1)$. As usual, $G[S]$ denotes the subgraph of a multigraph $G$ induced by a subset $S$ of $V(G)$, that is, $V(G[S])=S$ and $\mu_{G[S]}(x,y)=\mu_G(x,y)$ for all distinct $x,y \in S$. As in the case of graphs, a $1$-factor $F$ of a multigraph $G$ is a set of edges of $G$ such that each vertex of $G$ is incident with exactly one edge in $F$.

\subsection{Thickenings of multigraphs}

To study $1$-factor covers of graphs it is convenient to be able to `add further edges between pairs of already adjacent vertices'. The following concepts allow us to do this formally.

\begin{enumerate}
\item We say a multigraph $G'$ is a \emph{thickening} of a multigraph $G$ if $V(G')=V(G)$, $G$ is a subgraph of $G'$,  and, for all distinct $x,y \in V(G)$, $\mu_{G'}(x,y)=0$ if $\mu_{G}(x,y)=0$.

\item For a multigraph $G$ and a subset $S$ of $V(G)$, we say that a thickening $H$ of $G[S]$ is a \emph{$k$-thickening of $G$ on $S$} if, for each $x \in S$, $\deg_{H}(x)=k$ if $N_G(x) \subseteq S$, and $\deg_{H}(x) \leq k - \deg_G(x) + \deg_{G[S]}(x)$ otherwise.
\end{enumerate}

The following result is critical to our approach. Part (i) of it can be obtained from well known results in several ways. Here, we sketch a proof based on Theorem~\ref{T:infiniteHall}.

\begin{lemma}\label{L:infiniteKonig}
Let $k$ be a positive integer.
\begin{enumerate}
\item[\textup{(i)}]  The edge set of a $k$-regular bipartite multigraph can be partitioned into $k$ \hbox{$1$-factors}.
\item[\textup{(ii)}]  A bipartite (simple) graph $G$ has a $1$-factor cover with at most $k$ $1$-factors if and only if $G$ has a $k$-regular thickening.
\end{enumerate}
\end{lemma}

\begin{proof}
\textup{(i)}  Let $G^*$ be a $k$-regular bipartite multigraph with bipartition $\{V_1,V_2\}$. For any $T \in \P(V_1) \cup \P(V_2)$, the number of edges incident with vertices in $N_{G^*}(T)$ is
\[
k|N_{G^*}(T)|=\medop\sum_{x \in N_{G^*}(T)} \deg_{G^*}(x) \geq \medop\sum_{x \in T} \deg_{G^*}(x)=k|T|
\]
and hence $|N_{G^*}(T)|\geq|T|$. So by applying Theorem~\ref{T:infiniteHall} to $G^*$ (or, more precisely, to the unique simple graph of which ${G^*}$ is a thickening) we see that ${G^*}$ has a $1$-factor. Removing the edges of this $1$-factor from ${G^*}$ results in a $(k-1)$-regular bipartite multigraph, and so we can proceed inductively to prove part (i).

\textup{(ii)} Let $G$ be a simple bipartite graph. If $G$ has a $k$-regular thickening $G^*$ then, by part (i), there is a partition $\mathcal{F}^*=\{F^*_1,\ldots,F^*_k\}$ of the edge set of $G^*$ into $k$ \hbox{$1$-factors}. For each $i \in \{1,\ldots,k\}$, let $F_i$ be the $1$-factor of $G$ obtained by replacing each edge of $F^*_i$ with the edge of $G$ that is incident with the same two vertices. Then  $\mathcal{F}=\{F_1,\ldots,F_k\}$ is a 1-factor cover of $G$ with at most $k$ $1$-factors (note $F_1,\ldots,F_k$ may not all be distinct).

Suppose conversely that $G$ has a $1$-factor cover $\mathcal{F}=\{F_1,\dots,F_\ell\}$ with $\ell\leq k$. If $\ell<k$, define $F_k=\dots =F_{\ell+1}=F_\ell$. Define a thickening $G^*$ of $G$ by setting $\mu_{G^*}(x,y)=0$ if $x$ and $y$ are not adjacent in $G$ and, for each edge $\{x,y\}$ of $G$, setting $\mu_{G^*}(x,y)= |\{i \in \{1,\ldots,k\}:\{x,y\} \in F_i\}|$. Since each edge $\{x,y\}$ of $G$ lies in at least one of $F_1,\dots,F_\ell$, it follows that  $\mu_{G^*}(x,y)\geq 1=\mu_G(x,y)$, and hence $G^*$ is a thickening of $G$. Let $x$ be a vertex of $G$ and note that $\deg_{G^*}(x) =
\sum_{y\in N_G(x)}\mu_{G^*}(x,y)$. For each $i \in \{1,\ldots,k\}$, $F_i$ contains exactly one edge incident with $x$. Hence the sets $\{i \in \{1,\ldots,k\}:\{x,y\} \in F_i\}$, for  $y\in N_G(x)$, form a partition of $\{1,\ldots,k\}$. So $\deg_{G^*}(x) =\sum_{y\in N_G(x)}\mu_{G^*}(x,y)=k$ by our definition of $G^*$. Therefore $G^*$ is a $k$-regular thickening of $G$.
\end{proof}

We establish some basic properties of thickenings in our next result.

\begin{lemma}\label{L:basicThickeningProperties}
Let $G$ be a multigraph, let $S$ be a subset of $V(G)$, and let $S'$ be a subset of $S$.
\begin{itemize}
    \item[\textup{(i)}]
If $G^*$ is a $k$-regular thickening of $G$, then $G^*[S]$ is a $k$-thickening of $G$ on $S$.
    \item[\textup{(ii)}]
If $H$ is a $k$-thickening of $G$ on $S$, then $H[S']$ is a $k$-thickening of $G$ on $S'$.
\end{itemize}
\end{lemma}

\begin{proof}
(i)\quad Since $G^*$ is a thickening of $G$, its induced subgraph $G^*[S]$ is a thickening of $G[S]$. Let $x \in S$, and note that $\deg_{G^*}(x)=k$ because $G^*$ is $k$-regular.  Suppose first that $N_{G}(x) \subseteq S$. Then, because $G^*$ is a thickening of $G$, every edge of $G^*$ that is incident with $x$ is also in $G^*[S]$ and so $\deg_{G^*[S]}(x)=\deg_{G^*}(x)=k$. On the other hand suppose that $N_{G}(x) \nsubseteq S$. Then exactly $\deg_G(x) - \deg_{G[S]}(x)$ edges of $G$ incident with $x$ are not in $G[S]$ and hence, because $G^*$ is a thickening of $G$, at least this many of the $k$ edges of $G^*$ incident with $x$ are not in $G^*[S]$. Thus  $\deg_{G^*[S]}(x) \leq \deg_{G^*}(x) - \deg_G(x) + \deg_{G[S]}(x)=k- \deg_G(x) + \deg_{G[S]}(x)$, proving (i).

\medskip
(ii)\quad Since $H$ is a thickening of $G[S]$, its induced subgraph $H[S']$ is a thickening of $G[S']$. Let $x \in S'$. Suppose first that $N_{G}(x) \subseteq S'$. Then we have $\deg_{H[S']}(x)=\deg_H(x)$ and we know that $\deg_H(x)=k$ because $H$ is a $k$-thickening of $G$ on $S$ (and $S' \subseteq S$). On the other hand suppose that $N_{G}(x) \nsubseteq S'$. Exactly $\deg_{G[S]}(x) - \deg_{G[S']}(x)$ edges of $G[S]$ incident with $x$ are not in $G[S']$ and hence, because $H$ is a thickening of $G[S]$, at least this many edges of $H$ incident with $x$ are not in $H[S']$. Thus $\deg_{H[S']}(x) \leq \deg_{H}(x) - \deg_{G[S]}(x) + \deg_{G[S']}(x)$. Furthermore, since $H$ is a $k$-thickening of $G$ on $S$, $\deg_{H}(x) \leq k - \deg_G(x) + \deg_{G[S]}(x)$. Combining these inequalities we have $\deg_{H[S']}(x) \leq k - \deg_G(x) + \deg_{G[S']}(x)$, and (ii) is proved.
\end{proof}

Lemma~\ref{L:basicThickeningProperties} is the motivation for the definition of a `$k$-thickening on a set $S$'. It implies that having a $k$-thickening on $S$ for each $S \subseteq V(G)$ is a necessary condition for a multigraph $G$ to have a $k$-regular thickening. In our next result we show, in the style of the De Bruijn-Erd\H{o}s theorem \cite{deBErd}, that possessing this property on all finite vertex subsets is sufficient to guarantee that a multigraph has a $k$-regular thickening.

\begin{lemma}\label{L:finiteCondition}
Let $G$ be a bipartite (simple) graph and $k$ be a positive integer. Then $G$ has a $1$-factor cover with at most $k$ $1$-factors if and only if, for each finite subset $S$ of $V(G)$, there exists a $k$-thickening of $G$ on $S$.
\end{lemma}

\begin{proof}
Let $V=V(G)$. Suppose first that $G$ has a $1$-factor cover with at most $k$ 1-factors. Then, by Lemma~\ref{L:infiniteKonig}(ii), $G$ has a $k$-regular thickening $G^*$. Therefore, by Lemma~\ref{L:basicThickeningProperties}(i), for any $S \in \Pfin(V)$, $G^*[S]$ is a $k$-thickening of $G$ on $S$.

Conversely suppose that,  for every $S \in \Pfin(V)$, there exists a $k$-thickening of $G$ on $S$. Consider the set $\G$ of all thickenings $G'$ of $G$ with the property that there is a $k$-thickening of $G'$ on $S$ for every finite subset $S$ of $V$. Note that $\G$ is non-empty (since $G\in\G$), and each multigraph in $\G$ has maximum degree at most $k$ (for if $\deg_{G'}(x) > k$ for some multigraph $G'$ and $x \in V(G')$, then there is no $k$-thickening of $G'$ on $\{x\}$). Let $(\G,\leq)$ be the poset formed by $\G$ under subgraph inclusion. Let $\C$ be a chain in $(\G,\leq)$ and let $G_{\C}$ be the union of the multigraphs in $\C$, so $\mu_{G_\C}(x,y)=\max\{\mu_{G'}(x,y):G' \in \C\}$ for all distinct $x,y \in V$. In particular, if $\{x,y\}\not\in E(G)$, then $\mu_{G'}(x,y)=0$ for all $G' \in \C$ and hence $\mu_{G_\C}(x,y)=0$. It follows that $G_{\C}$ is a thickening of $G$. We will show that $G_{\C} \in \G$ and hence that $G_{\C}$ is an upper bound for $\C$
in $(\G,\leq)$.

Let $S \in \Pfin(V)$. Because $S$ is finite and $G$ has maximum degree at most $k$, the set $E_S$ of edges of $G$ that are incident with at least one vertex in $S$ is finite. By definition of $G_{\C}$, for each $\{y,z\} \in E_S$, there is some $G'_{\{y,z\}} \in \C$ such that $\mu_{G'_{\{y,z\}}}(y,z)=\mu_{G_\C}(y,z)$. So, since $\C$ is a chain, there exists a single $G'_S \in \C$ such that $\mu_{G'_S}(y,z)=\mu_{G_\C}(y,z)$ for all $\{y,z\} \in E_S$. Because $G'_S$ and $G_\C$ are thickenings of $G$, $\mu_{G'_S}(y,z)=\mu_{G_\C}(y,z)=0$ for all distinct $y,z \in V$ such that $\{y,z\} \notin E(G)$. Now, because $G'_S \in \G$, there is a $k$-thickening $H'$ of $G'_S$ on $S$. It follows from our definition of $G'_S$ that $H'$ is also a $k$-thickening of $G_{\C}$ on $S$. Thus $G_{\C} \in \G$ and $G_{\C}$ is an upper bound for $\C$ in $(\G,\leq)$. So every chain in $(\G,\leq)$ has an upper bound, and  by Zorn's lemma, $(\G,\leq)$ contains a maximal element, say $G^*$. We claim that $G^*$ is a $k$-regular thickening of $G$. Note that, if this is true, then by Lemma~\ref{L:infiniteKonig}(ii), $G$ has a $1$-factor cover of $G$ with at most $k$ $1$-factors. Since $G^*$ is a thickening of $G$ (by the definition of $\G$), it only remains to prove that $G^*$ is $k$-regular.

As noted above, since $G^*\in\G$, each vertex of $G^*$ has degree at most $k$. Suppose for a contradiction that $G^*$ has a vertex $x$ with $\deg_{G^*}(x) < k$. Let $N_G(x)=\{y_1,\ldots,y_t\}$, where $t < k$, and for each $i \in \{1,\ldots,t\}$, let $G^*_i$ be the multigraph obtained from $G^*$ by adding one additional edge between $x$ and $y_i$. Because $G^*$ is maximal in $(\G,\leq)$, for each $i \in \{1,\ldots,t\}$ we have $G^*_i \notin \G$ and hence for some $S_i \in \Pfin(V)$ there is no $k$-thickening of $G^*_i$ on $S_i$. Let $S=N_G(x) \cup S_1 \cup \cdots \cup S_t$, and note that $S$ is finite. Because $G^* \in \G$, there is a $k$-thickening $H^*$ of $G^*$ on $S$. By definition of a `$k$-thickening of $G^*$ on $S$' it follows that $\deg_{H^*}(x)=k$ since $N_{G^*}(x)=N_G(x)\subseteq S$. Thus, for some $j \in \{1,\ldots,t\}$, we have $\mu_{H^*}(x,y_j)\geq \mu_{G^*}(x,y_j)+1=\mu_{G^*_j}(x,y_j)$, and so  $H^*$ is a $k$-thickening of $G^*_j$ on $S$. But then, by Lemma~\ref{L:basicThickeningProperties}(ii), $H^*[S_j]$ is a $k$-thickening of $G^*_j$ on $S_j$, and this is a contradiction to the definition of $S_j$. So $G^*$ is indeed $k$-regular and the proof is complete.
\end{proof}

\subsection{Flow networks and the proof of Theorem~\ref{T:graphkCharacterisation}}\label{flow}

Following \cite{CarRiz}, our proof of Theorem~\ref{T:graphkCharacterisation} will make use of flow networks. We give the basic definitions here and refer the reader to \cite[Chapter 8]{KV} for a more detailed treatment.

A \emph{flow network} is a finite digraph where every arc has a nonnegative \emph{capacity} associated with it and where two special vertices are distinguished as a \emph{source} and a \emph{sink} such that the source has no arcs into it and the sink has no arcs out from it. A \emph{flow} in such a network is an assignment of a nonnegative value to each arc such that no value exceeds the capacity of its arc and, at each vertex other than the source and sink, the total flow in equals the total flow out. The sum of the flows on arcs emerging from the source is the \emph{magnitude} of the flow (and will necessarily be equal to the sum of the flows on arcs going into the sink). A \emph{cut} in such a network is a bipartition $(A,B)$ of the vertices with the source in $A$ and the sink in $B$. The \emph{capacity of a cut $(A,B)$} is the total capacity of the arcs that emerge from vertices in $A$ and go into vertices in $B$. The \emph{max flow min cut theorem} states that the maximum magnitude of the flow through such a network is equal to the minimum capacity of a cut over all possible cuts of the network. Furthermore, the \emph{integer flow theorem} states that if a flow network has integer capacities on all of its arcs then it has an integer-valued maximum flow.

\begin{proof}[\textup{\textbf{Proof of Theorem~\ref{T:graphkCharacterisation}.}}]
Let $V=V(G)$ and $\{V_1,V_2\}$ be the bipartition of $V$.

We first prove the `only if' direction. Suppose that $G$ has a 1-factor cover with at most $k$ $1$-factors. Because a $1$-factor contains exactly one edge incident with a given vertex, condition (i) must hold. For any set $T \in \Pfin(V_1) \cup \Pfin(V_2)$, each $1$-factor of $G$ contains $|T|$ edges between vertices in $T$ and vertices in $N_G(T)$ and hence contains $|N_G(T)|-|T|$ edges that are incident with a vertex in $N_G(T)$ but not with a vertex in $T$. Since there are at most $k$ $1$-factors, there are at most $k(|N_G(T)|-|T|)$ edges of $G$ that are incident with a vertex in $N_G(T)$ but not with a vertex in $T$. On the other hand, there are exactly $\sum_{x \in N_G(T)} \deg_G(x) - \sum_{x \in T}\deg_G(x)$ such edges (noting that no edge is incident with two vertices in $N_G(T)$ since $N_G(T) \subseteq V_1$ or $N_G(T) \subseteq V_2$).  Thus condition (ii) holds.

So it remains to prove the `if' direction.  Fix $k$ and suppose that conditions (i) and (ii) of Theorem~\ref{T:graphkCharacterisation} hold for $G$. We will  use Lemma~\ref{L:finiteCondition} to show that  $G$ has a $1$-factor cover with at most $k$ $1$-factors. Let $S^* \in \Pfin(V)$. We wish to find a $k$-thickening of $G$ on $S^*$. We may assume that $S^* \cap V_i \neq \emptyset$ for each $i \in \{1,2\}$ for otherwise $G[S^*]$ is an empty graph and so the empty graph on $S^*$ is the unique $k$-thickening of $G$ on $S^*$. Let $S_1=S^* \cap V_1$, let $S= S^* \cup N_G(S_1)$,
let $S_2=S \cap V_2$, and note that $S$ is finite and $S_1=S\cap V_1$. It suffices to find a $k$-thickening $H$ of $G$ on $S$ because $S$ contains $S^*$ and, by Lemma~\ref{L:basicThickeningProperties}(ii), $H[S^*]$ will be a  $k$-thickening of $G$ on $S^*$.

For each $x \in S$, let $c_x=k-\deg_G(x)$ and note that $c_x \geq 0$ because (i) holds. Because (ii) holds, for any $T \in \Pfin(S_1) \cup \Pfin(S_2)$ we have,
\begin{equation}\label{E:bTIneq}
k\big(|N_G(T)|-|T|\big) \geq \sum_{x\in N_G(T)}\deg_G(x)-\medop\sum_{x\in T}\deg_G(x).
\end{equation}
Let $S''_2=\{y \in S_2:N_G(y)\subseteq S_1\}$ and $S'_2=S_2 \setminus S''_2$.  If either $T\subseteq S_1$ or $T\subseteq S_2''$, then $N_G(T)\subseteq S$, and hence \eqref{E:bTIneq} yields
\begin{equation}\label{E:capacityIneq}
\medop\sum_{x \in N_G(T)}c_x \geq \medop\sum_{x\in T}c_x.
\end{equation}
Let $m=\sum_{x \in S_1}c_x$ and $m''=\sum_{y \in S''_2}c_y$. Because \eqref{E:capacityIneq} holds with $T=S''_2$ and because $N_G(S''_2) \subseteq S_1$, we have $m'' \leq \sum_{x \in N_G(S''_2)}c_x \leq m$. We will show that a $k$-thickening of $G$ on $S$ must exist if there is an integer flow of magnitude $m$ through the flow network $\D$ defined as follows.
\begin{itemize}[nosep]
    \item $\D$ has \emph{vertex set} the disjoint union $S\cup \{a,b,b'\}$, where
$a$ is the source, $b$ is the sink, each vertex of $S$ is an internal vertex of $\D$, and $b'$ is an additional internal vertex.
    \item The \emph{arcs} of $\D$ and their \emph{capacities} are as follows:
\begin{itemize}[nosep]
\item for each $x \in S_1$, $y \in S_2$ with $\{x,y\} \in E(G)$, $(x,y)$ is an arc with infinite capacity;
    \item for each $x \in S_1$, $(a,x)$ is an arc with capacity $c_x$;
    \item  for each $y \in S'_2$, $(y,b')$ is an arc with capacity $c_y$;
    \item  for each $y \in S''_2$, $(y,b)$ is an arc with capacity $c_y$;
    \item $(b',b)$ is an arc with capacity $m-m''$.
\end{itemize}
\end{itemize}
%%%%%%%%%%%%%%%%%%%%%%%%%%%%%%%%%%%%%%%%%%%%%%%%%%%%%%%%%%%%%%%%%%%%%%%%%%%%%%%%%%%%%%%%%%%%%%%%%%%%%%%%%%%%%%%%%%%%%%%%%%%%%%%%%%%%%%%%%%%%%%%%%%%%%%%%%%%%%%%%%%%%%%%%%%%%%%
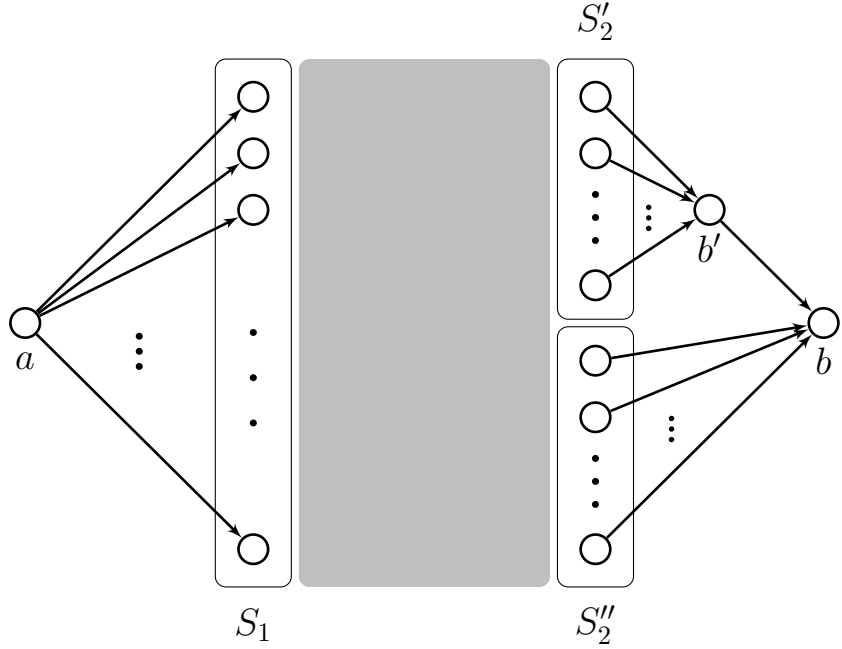
\begin{figure}[H]
\begin{center}
\begin{tikzpicture}[scale=1.0]
\tikzset{vertex/.style = {shape=circle,draw, line width=1pt,opacity=1.0}}
\tikzset{arc/.style = {->,> = latex', line width=1pt,opacity=1.0}}
\tikzset{edge/.style = {-,> = latex', line width=1pt,opacity=1.0}}
\draw[rounded corners] (-0.5,3.5) rectangle (0.5,-3.5);%we can replace this line by the next line
%\draw[rounded corners, left color=white,right color=gray] (-0.5,3.5) rectangle (0.5,-3.5);
\node  at (0,-4) {\large{$S_1$}};
\draw[rounded corners] (4,3.5) rectangle (5,0.05);%we can replace this line by the next line
%\shade[rounded corners, left color=gray,right color=gray] (4,3.5) rectangle (5,0.1);
\node  at (4.5,4) {\large{$S'_2$}};
\draw[rounded corners] (4,-0.05) rectangle (5,-3.5);%we can replace this line by the next line
%\shade[rounded corners, left color=gray,right color=gray] (4,-0.1) rectangle (5,-3.5);
\node  at (4.5,-4) {\large{$S''_2$}};
\shade[rounded corners, left color=gray!50!white,right color=gray!50!white] (0.6,3.5) rectangle (3.9,-3.5);
%\node  at (2.25,-4) {$G$};%we can label $G$ the gray rectangle
\node[vertex] (a) at (-3,0) {};
\node  at (-3,-0.5) {\large{$a$}};
\node[vertex] (b) at (0,3) {};
%\node  at (0,2.5) {$w_{-2}$};
\node[vertex] (c) at  (0,2.25) {};
\node[vertex] (n) at  (0,1.5) {};
%\node  at (-2,-0.5) {$w_{-1}$};
\node[vertex] (d) at  (0,-3) {};
%\node  at (-1,-0.5) {$w_{0}$};
\node[vertex] (e) at  (4.5,3) {};
%\node  at (0,-0.5) {$w_{1}$};
\node[vertex] (f) at  (4.5,2.25) {};
%\node  at (1,-0.5) {$w_{2}$};
\node[vertex] (g) at (4.5,0.5) {};
%\node  at (2,-0.5) {$w_{3}$};
\node[vertex] (h) at (4.5,-0.5) {};
%\node  at (3,-0.5) {$w_{4}$};
\node[vertex] (i) at (4.5,-1.25) {};
%\node  at (4,-0.5) {$w_{5}$};
\node[vertex] (j) at (4.5,-3) {};
%\node  at (2,-0.5) {$w_{3}$};
\node[vertex] (k) at (6,1.5) {};
\node  at (6,1) {\large{$b'$}};
\node[vertex] (m) at (7.5,0) {};
\node  at (7.5,-0.5) {\large{$b$}};

\draw[arc] (a)  to (b);
\draw[arc] (a) to (c);
\draw[arc] (a) to (n);
\draw[arc] (a) to (d);
\draw[arc] (e) to (k);
\draw[arc] (f)  to (k);
\draw[arc] (g) to (k);
\draw[arc] (h) to (m);
\draw[arc] (i)  to (m);
\draw[arc] (j) to (m);
\draw[arc] (k) to (m);
%\node  at (0,0.45) {\huge{$\cdot$}};
\node  at (0,-0.15) {\huge{$\cdot$}};
\node  at (0,-0.75) {\huge{$\cdot$}};
\node  at (0,-1.35) {\huge{$\cdot$}};
%\node  at (0,-1.95) {\huge{$\cdot$}};

\node  at (-1.5,-0.2) {\huge{$\cdot$}};
\node  at (-1.5,-0.4) {\huge{$\cdot$}};
\node  at (-1.5,-0.6) {\huge{$\cdot$}};

\node  at (5.2,1.5) {\huge{$\vdots$}};
\node  at (5.5,-1.3) {\huge{$\vdots$}};

\node  at (4.5,1.68) {\huge{$\cdot$}};
\node  at (4.5,1.375) {\huge{$\cdot$}};
\node  at (4.5,1.07) {\huge{$\cdot$}};

\node  at (4.5,-1.82) {\huge{$\cdot$}};
\node  at (4.5,-2.125) {\huge{$\cdot$}};
\node  at (4.5,-2.43) {\huge{$\cdot$}};

\end{tikzpicture}
\caption{The flow network $\mathcal{D}$ from the proof of Theorem~\ref{T:graphkCharacterisation}. Each arc entering a vertex $x\in S_1$ has capacity $c_x$, each arc exiting a vertex $y \in S'_2 \cup S''_2$ has capacity $c_y$, and the arc from $b'$ to $b''$ has capacity $m-m''$. The arcs from $S_1$ to $S'_2 \cup S''_2$ correspond to edges in $G[S]$ and each has infinite capacity.}
\label{flownet}
\end{center}
\end{figure}
%%%%%%%%%%%%%%%%%%%%%%%%%%%%%%%%%%%%%%%%%%%%%%%%%%%%%%%%%%%%%%%%%%%%%%%%%%%%%%%%%%%%%%%%%%%%%%%%%%%%%%%%%%%%%%%%%%%%%%%%%%%%%%%%%%%%%%%%%%%%%%%%%%%%%%%%%%%%%%%%%%%%%%%%%%%%%%
Observe that, since $\sum_{x\in S_1}c_x=m$, any flow of magnitude $m$ through $\D$ must use each arc incident with $a$ at full capacity.  Similarly, such a flow must use each arc incident with $b$ at full capacity because $m=(m-m'')+\sum_{y \in S''_2}c_y$.
For any integer flow of magnitude $m$ through $\D$ we associate the $k$-thickening $H$ of $G$ on $S$ obtained from $G[S]$ by adding, for each $x \in S_1$ and $y \in S_2$ such that $\{x,y\} \in E(G)$, $i_{xy}$ further edges between $x$ and $y$, where $i_{xy}$ is the flow along the arc $(x,y)$ in $\D$. To see that $H$ is indeed a $k$-thickening of $G$ on $S$, note the following (recalling the definition of a flow given immediately before this proof).

\begin{itemize}[nosep]
    \item
For each $x \in S_1$, $N_G(x)\subseteq S_2 \subseteq S$ and the total flow into $x$ is $c_x$ because the arc $ax$ is used at full capacity. So the total flow out of $x$, $i_x=\sum_{y \in N_G(x)}i_{xy}$, must be $c_x$. Thus, since $N_G(x)\subseteq S$, we have $\deg_H(x)=\deg_{G[S]}(x)+i_x = \deg_G(x)+c_x =k$.
    \item
For each $y \in S''_2$, $N_G(y)\subseteq S_1\subseteq S$ and the total flow out of $y$ is $c_y$ because the arc $yb$ is used at full capacity. So the total flow into $y$, $i_y=\sum_{x \in N_G(y)}i_{xy}$, must be $c_y$. Thus, since $N_G(y)\subseteq S$, we have $\deg_H(y)=\deg_{G[S]}(y)+i_y = \deg_G(y)+c_y =k$.
    \item
For each $y \in S'_2$, the total flow out of $y$ is at most $c_y$, the capacity of the arc  $yb'$. So the total flow into $y$, $i_y=\sum_{x\in N_G(y)\cap S_1}i_{xy}$, must be at most $c_y$. Thus we have $\deg_H(y)=\deg_{G[S]}(y)+i_y \leq \deg_{G[S]}(y)+c_y =k-\deg_G(y)+\deg_{G[S]}(y)$.
\end{itemize}

So to finish the proof, it suffices to show that there exists an integer flow of magnitude $m$ on $\D$. Since $m$ is the maximum
magnitude for a flow on $\D$, it follows from the max-flow min-cut theorem, and the integer flow theorem, that it suffices to show that the minimum capacity  across a cut of $\D$ is exactly $m$. Let $(A,B)$ be a minimum capacity cut of $\D$, where $a \in A$ and $b \in B$, and let $c(A,B)$ be the capacity across $(A,B)$. Let $A_1=A \cap S_1$, $B_1=B \cap S_1$, $A'_2=A \cap S'_2$, $B'_2=B \cap S'_2$, $A''_2=A \cap S''_2$ and $B''_2=B \cap S''_2$. The cut $(\{a\},V(\D) \setminus \{a\})$ of $\D$ has capacity $m$ across it, so $c(A, B) \leq m$ and in particular $c(A, B)$ is finite. Thus, in $\D$, there is no arc of infinite capacity from a vertex in $A$ to a vertex in $B$. So, by the construction of $\D$, there is no edge $\{x,y\}$ in $G$ with $x \in A_1$ and $y \in  B_2'\cup B_2''$. Thus, because $N_G(S_1) \subseteq S_2$ and $N_G(S''_2) \subseteq S_1$, we have $N_G(A_1) \subseteq A'_2 \cup A''_2$.  Also, since $N_G(B_2'')\subseteq S_1=A_1\cup B_1$, this implies that $N_G(B''_2) \subseteq B_1$. We will complete the proof by showing that $c(A,B)-m$ is nonnegative, considering two cases according to whether or not $b' \in A$.

First suppose that $b' \notin A$. Then
\begin{align*}
c(A,B)-m &=\left(\medop\sum_{x \in B_1}c_x \right) + \left(\medop\sum_{y \in A'_2 \cup A''_2}c_y \right) - m \\[1mm]
&= \medop\sum_{y \in A'_2 \cup A''_2}c_y-\medop\sum_{x \in A_1}c_x
\end{align*}
where the sums in the first line come from arcs from $a$ to vertices in $B_1$, and arcs from vertices in $A'_2\cup A_2''$ to vertices in $\{b,b'\}$, and where the last equality follows by applying the definition of $m$. Because \eqref{E:capacityIneq} holds with $T=A_1$ and because $N_G(A_1) \subseteq A'_2 \cup A''_2$, this last expression is nonnegative.

Now suppose that $b' \in A$. Then
\begin{align*}
c(A,B)-m &= \left(\medop\sum_{x \in B_1}c_x\right) + \left(\medop\sum_{y \in A''_2}c_y\right) + (m-m'')- m \\[1mm]
&= \medop\sum_{x \in B_1}c_x-\medop\sum_{y \in B''_2}c_y
\end{align*}
where the positive terms in the first line come from arcs from $a$ to vertices in $B_1$, arcs from vertices in $A_2''$ to $b$, and the arc from $b'$ to $b$, and where the last equality follows by applying the definition of $m''$. Because \eqref{E:capacityIneq} holds with $T=B''_2$ and because $N_G(B''_2) \subseteq B_1$, this last expression is nonnegative.
\end{proof}

\section{Low degree digraphs requiring many derangements}\label{S:examples}

Considering Theorem~\ref{T:infiniteHall} and Lemma~\ref{T:infinite1Extendible}, it is clear where we should look for  examples of bipartite graphs that have no $1$-factor, and examples of bipartite graphs that have a $1$-factor but no $1$-factor cover,
and indeed it is not hard to find such examples. Correspondingly, considering Theorem~\ref{T:infinite1ExtendibleDigraph}, one can find examples of digraphs that cannot be generated by any (infinite or finite) set of derangements. Further, Corollary~\ref{C:regDigraph} makes it easy to find examples of digraphs that require $k$ derangements to generate them for any positive integer $k$, but in these examples
each in- and out-degree will be equal to $k$. Here we present examples of bipartite graphs with low maximum degree that do possess $1$-factor covers but for which the number of $1$-factors in any $1$-factor cover must be arbitrarily large, or infinite. These lead to examples of digraphs with low maximum in- and out-degree whose generating sets of derangements must be arbitrarily large or infinite.

\begin{example}\label{Ex:graph}
Let $G$ be the graph with vertex set $V=\{u_i: i\in\mathbb{Z}\} \cup \{v_i: i\in\mathbb{Z}\}$ and edge set
\[E=\{\{v_{2i+1},u_{2i+1}\}, \{v_i,v_{i+1}\}, \{u_i,u_{i+1}\}: i\in\mathbb{Z}\}\]
depicted in Figure \ref{infinitebadness}. Note that $G$ is connected, has maximum degree $3$, and is bipartite with bipartition $\{\{u_{2i+1},v_{2i}: i \in \mathbb{Z}\},\{u_{2i},v_{2i+1}: i \in \mathbb{Z}\}\}$. For any positive integer $n$, consider the finite subset $T=\{v_{2i}:\ 1\leq i\leq n\}$ of $V$. We have $N_G(T)=\{v_{2i+1}: 0 \leq i \leq n\}$ so $|N_G(T)|=n+1$, $|T|=n$, and $\sum_{x\in N_G(T)}\deg_G(x)-\sum_{x\in T}\deg_G(x)=3(n+1)-2n=n+3$. Since $n$ can be chosen to be any positive integer, by Theorem~\ref{T:graphkCharacterisation}(ii), $G$ does not have a $1$-factor cover with finitely many $1$-factors. Intuitively, this is because each one factor of $G$ contains at most one of the vertical edges $\{\{v_{2i+1},u_{2i+1}\}: i \in \mathbb{Z}\}$. However, by Lemma~\ref{T:infinite1Extendible} or by simple inspection, it can be seen that $G$ is $1$-extendable and so does admit a $1$-factor cover with infinitely many $1$-factors.

In addition, for each positive integer $k$, $G$ has a finite subgraph that requires exactly $k$ $1$-factors to cover it: consider the induced bipartite subgraph $G_k$ of $G$ with vertex set $V_k=\{v_i,u_i: 1\leq i\leq 2k-1\}$. It is readily seen that $G_k$ has a unique $1$-factor cover with $k$ $1$-factors and, by applying Theorem~\ref{T:graphkCharacterisation} with $T=\{v_{2i}:\ 1\leq i\leq k-1\}$, we see that $G_k$ does not admit a $1$-factor cover with fewer than $k$ $1$-factors.
\end{example}
\begin{figure}[H]
\begin{center}
\begin{tikzpicture}[scale=1.0]
\tikzset{vertex/.style = {shape=circle,draw, line width=1pt,opacity=1.0}}
\tikzset{edge/.style = {-,> = latex', line width=1pt,opacity=1.0}}
\node[vertex] (a) at (-4,0) {};
\node  at (-4,-0.5) {$v_1$};
\node[vertex] (b) at (-3,0) {};
\node  at (-3,-0.5) {$v_2$};
\node[vertex] (c) at  (-2,0) {};
\node  at (-2,-0.5) {$v_3$};
\node[vertex] (d) at  (-1,0) {};
\node  at (-1,-0.5) {$v_4$};
\node[vertex] (e) at  (0,0) {};
\node  at (0,-0.5) {$v_5$};
\node[vertex] (f) at  (1,0) {};
\node  at (1,-0.5) {$v_6$};
\node[vertex] (g) at (2,0) {};
\node  at (2,-0.5) {$v_7$};
\node[vertex] (h) at (3,0) {};
\node  at (3,-0.5) {$v_8$};
\node[vertex] (i) at (4,0) {};
\node  at (4,-0.5) {$v_9$};
\node[vertex] (j) at (-4,1) {};
\node  at (-4,1.5) {$u_1$};
\node[vertex] (l) at (-3,1) {};
\node  at (-3,1.5) {$u_2$};
\node[vertex] (m) at (-2,1) {};
\node  at (-2,1.5) {$u_3$};
\node[vertex] (n) at (-1,1) {};
\node  at (-1,1.5) {$u_4$};
\node[vertex] (o) at (0,1) {};
\node  at (0,1.5) {$u_5$};
\node[vertex] (p) at (1,1) {};
\node  at (1,1.5) {$u_6$};
\node[vertex] (q) at (2,1) {};
\node  at (2,1.5) {$u_7$};
\node[vertex] (r) at (3,1) {};
\node  at (3,1.5) {$u_8$};
\node[vertex] (s) at (4,1) {};
\node  at (4,1.5) {$u_9$};

\node  at (-5.3,0) {\large{$\cdots$}};
\draw[edge] (-5,0) to[out=0,in=180] (a);
\draw[edge] (a)  to[out=0,in=180] (b);
\draw[edge] (b)  to [out=0,in=180] (c);
\draw[edge] (c)  to [out=0,in=180] (d);
\draw[edge] (d)  to [out=0,in=180] (e);
\draw[edge] (e)  to [out=0,in=180] (f);
\draw[edge] (f)  to [out=0,in=180] (g);
\draw[edge] (g)  to [out=0,in=180] (h);
\draw[edge] (h)  to [out=0,in=180] (i);
\draw[edge] (i) to[out=0,in=180] (5,0);
\node  at (5.5,1) {\large{$\cdots$}};
\node  at (-5.3,1) {\large{$\cdots$}};
\draw[edge] (-5,1) to[out=0,in=180] (j);
\draw[edge] (j)  to[out=0,in=180] (l);
\draw[edge] (l)  to [out=0,in=180] (m);
\draw[edge] (m)  to [out=0,in=180] (n);
\draw[edge] (n)  to [out=0,in=180] (o);
\draw[edge] (o)  to [out=0,in=180] (p);
\draw[edge] (p)  to [out=0,in=180] (q);
\draw[edge] (q)  to [out=0,in=180] (r);
\draw[edge] (r)  to [out=0,in=180] (s);
\draw[edge] (a)  to [out=90,in=270] (j);
\draw[edge] (c)  to [out=90,in=270] (m);
\draw[edge] (e)  to [out=90,in=270] (o);
\draw[edge] (g)  to [out=90,in=270] (q);
\draw[edge] (i)  to [out=90,in=270] (s);
\draw[edge] (s) to[out=0,in=180] (5,1);
\node  at (5.5,0) {\large{$\cdots$}};
\end{tikzpicture}
\caption{Bipartite graph in Example~\ref{Ex:graph} that can be covered with infinitely many $1$-factors but not with finitely many.}
\label{infinitebadness}
\end{center}
\end{figure}
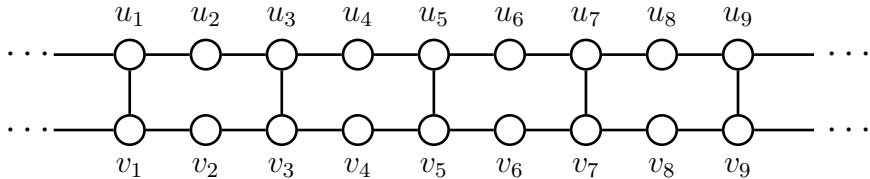

\begin{example}\label{Ex:graph2}
Example~\ref{Ex:graph} can be generalised to produce, for any integer $k\geq 1$, a bipartite graph with maximum degree $k+2$ that can be covered with infinitely many 1-factors but not with finitely many. Consider the cartesian product $P_\infty\Box H$ where $H$ is a finite $k$-regular bipartite graph and $P_\infty$ is the infinite path with vertex set $\mathbb{Z}$ and edge set $\{\{i,i+1\}:i \in \mathbb{Z}\}$. Thus the edge set of $P_\infty\Box H$ is
\[
\bigl\{ \{ (i,y), (i+1,y)\}:i \in \mathbb{Z}, y\in V(H)\bigr\} \ \cup \  \bigl\{ \{ (i,y), (i, z)\}:i \in \mathbb{Z}, \{ y, z\}\in E(H)\bigr\}.
\]
Now subdivide each edge $\{(i,y),(i+1,y)\}$ such that $i \in \mathbb{Z}$ and $y \in V(H)$ with a new vertex $(i^*,y)$. The resulting graph $G$ has maximum degree $k+2$, and has a $1$-factor cover with infinitely many $1$-factors which we describe in the next paragraph. However, applying Theorem~\ref{T:graphkCharacterisation} with $T=\{(i^*,y):1 \leq i \leq n\}$ for some fixed $y \in V(H)$ and arbitrarily large positive integer $n$ implies that $G$ cannot be covered with finitely many 1-factors. When $k=1$ and $H$ is the complete graph $K_2$, we recover Example~\ref{Ex:graph}.

We can construct an infinite $1$-factor cover of $G$ as follows. For each $i \in \mathbb{Z}$, let $H_i$ be the subgraph of $G$ induced by the set $V_i=\{(i,y): y\in V(H)\}$ and let $F_i$ be the unique 1-factor of $G \setminus V_i$, where $G \setminus V_i$ is the graph obtained from $G$ by deleting the vertices in $V_i$ (see Figure \ref{uni1fac}). For each $i \in \mathbb{Z}$, $H_i$ is isomorphic to $H$ and, by Lemma~\ref{L:infiniteKonig}(i), there is a partition $\{H_{i,1},\ldots,H_{i,k}\}$ of $E(H_i)$ into $k$ $1$-factors. Then $F_{i,j}=F_i\cup H_{i,j}$ is a 1-factor of $G$ for any $i\in\mathbb{Z}$ and $j \in \{1,\ldots,k\}$ and it can be seen that $\{F_{i,j}: i\in\mathbb{Z}, 1\leq j\leq k\}$ is a 1-factor cover of $G$.
\end{example}

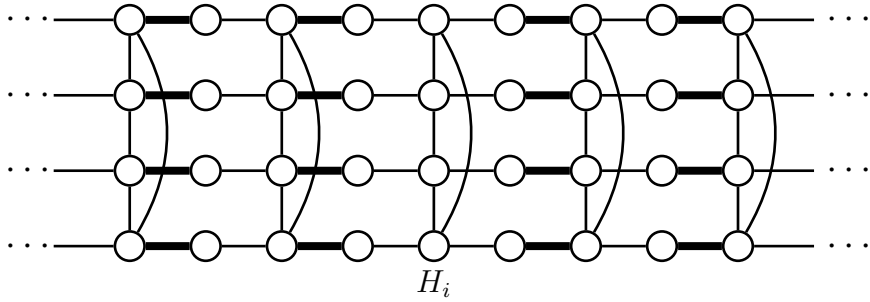
\begin{figure}[H]
\begin{center}
\begin{tikzpicture}[scale=1.0]
\tikzset{vertex/.style = {shape=circle,draw, line width=1pt,opacity=1.0}}
\tikzset{edge/.style = {-,> = latex', line width=1pt,opacity=1.0}}
\tikzset{edge1/.style = {-,> = latex', line width=3pt,opacity=1.0}}

\node[vertex] (a) at (-4,0) {};
\node[vertex] (b) at (-3,0) {};
\node[vertex] (c) at  (-2,0) {};
\node[vertex] (d) at  (-1,0) {};
\node[vertex] (e) at  (0,0) {};
\node[vertex] (f) at  (1,0) {};
\node[vertex] (g) at (2,0) {};
\node[vertex] (h) at (3,0) {};
\node[vertex] (i) at (4,0) {};
\node[vertex] (j) at (-4,1) {};
\node[vertex] (l) at (-3,1) {};
\node[vertex] (m) at (-2,1) {};
\node[vertex] (n) at (-1,1) {};
\node[vertex] (o) at (0,1) {};
\node[vertex] (p) at (1,1) {};
\node[vertex] (q) at (2,1) {};
\node[vertex] (r) at (3,1) {};
\node[vertex] (s) at (4,1) {};

\node[vertex] (a1) at (-4,2) {};
\node[vertex] (b1) at (-3,2) {};
\node[vertex] (c1) at  (-2,2) {};
\node[vertex] (d1) at  (-1,2) {};
\node[vertex] (e1) at  (0,2) {};
\node  at (0,-0.5) {$H_i$};
\node[vertex] (f1) at  (1,2) {};
\node[vertex] (g1) at (2,2) {};
\node[vertex] (h1) at (3,2) {};
\node[vertex] (i1) at (4,2) {};
\node[vertex] (j1) at (-4,3) {};
\node[vertex] (l1) at (-3,3) {};
\node[vertex] (m1) at (-2,3) {};
\node[vertex] (n1) at (-1,3) {};
\node[vertex] (o1) at (0,3) {};
\node[vertex] (p1) at (1,3) {};
\node[vertex] (q1) at (2,3) {};
\node[vertex] (r1) at (3,3) {};
\node[vertex] (s1) at (4,3) {};

\node  at (-5.3,0) {\large{$\cdots$}};
\draw[edge] (-5,0) to[out=0,in=180] (a);
\draw[edge1] (a)  to[out=0,in=180] (b);
\draw[edge] (b)  to [out=0,in=180] (c);
\draw[edge1] (c)  to [out=0,in=180] (d);
\draw[edge] (d)  to [out=0,in=180] (e);
\draw[edge] (e)  to [out=0,in=180] (f);
\draw[edge1] (f)  to [out=0,in=180] (g);
\draw[edge] (g)  to [out=0,in=180] (h);
\draw[edge1] (h)  to [out=0,in=180] (i);
\draw[edge] (i) to[out=0,in=180] (5,0);
\node  at (5.5,1) {\large{$\cdots$}};
\node  at (-5.3,1) {\large{$\cdots$}};
\draw[edge] (-5,1) to[out=0,in=180] (j);
\draw[edge1] (j)  to[out=0,in=180] (l);
\draw[edge] (l)  to [out=0,in=180] (m);
\draw[edge1] (m)  to [out=0,in=180] (n);
\draw[edge] (n)  to [out=0,in=180] (o);
\draw[edge] (o)  to [out=0,in=180] (p);
\draw[edge1] (p)  to [out=0,in=180] (q);
\draw[edge] (q)  to [out=0,in=180] (r);
\draw[edge1] (r)  to [out=0,in=180] (s);
\draw[edge] (a)  to [out=90,in=270] (j);
\draw[edge] (c)  to [out=90,in=270] (m);
\draw[edge] (e)  to [out=90,in=270] (o);
\draw[edge] (g)  to [out=90,in=270] (q);
\draw[edge] (i)  to [out=90,in=270] (s);
\draw[edge] (s) to[out=0,in=180] (5,1);
\node  at (5.5,0) {\large{$\cdots$}};

\node  at (-5.3,2) {\large{$\cdots$}};
\draw[edge] (-5,2) to[out=0,in=180] (a1);
\draw[edge1] (a1)  to[out=0,in=180] (b1);
\draw[edge] (b1)  to [out=0,in=180] (c1);
\draw[edge1] (c1)  to [out=0,in=180] (d1);
\draw[edge] (d1)  to [out=0,in=180] (e1);
\draw[edge] (e1)  to [out=0,in=180] (f1);
\draw[edge1] (f1)  to [out=0,in=180] (g1);
\draw[edge] (g1)  to [out=0,in=180] (h1);
\draw[edge1] (h1)  to [out=0,in=180] (i1);
\draw[edge] (i1) to[out=0,in=180] (5,2);
\node  at (5.5,3) {\large{$\cdots$}};
\node  at (-5.3,3) {\large{$\cdots$}};
\draw[edge] (-5,3) to[out=0,in=180] (j1);
\draw[edge1] (j1)  to[out=0,in=180] (l1);
\draw[edge] (l1)  to [out=0,in=180] (m1);
\draw[edge1] (m1)  to [out=0,in=180] (n1);
\draw[edge] (n1)  to [out=0,in=180] (o1);
\draw[edge] (o1)  to [out=0,in=180] (p1);
\draw[edge1] (p1)  to [out=0,in=180] (q1);
\draw[edge] (q1)  to [out=0,in=180] (r1);
\draw[edge1] (r1)  to [out=0,in=180] (s1);
\draw[edge] (a1)  to [out=90,in=270] (j1);
\draw[edge] (c1)  to [out=90,in=270] (m1);
\draw[edge] (e1)  to [out=90,in=270] (o1);
\draw[edge] (g1)  to [out=90,in=270] (q1);
\draw[edge] (i1)  to [out=90,in=270] (s1);
\draw[edge] (s1) to[out=0,in=180] (5,3);
\node  at (5.5,2) {\large{$\cdots$}};

\draw[edge] (o)  to[out=90,in=270] (e1);
\draw[edge] (e)  to[out=60,in=300] (o1);
\draw[edge] (q)  to[out=90,in=270] (g1);
\draw[edge] (g)  to[out=60,in=300] (q1);
\draw[edge] (s)  to[out=90,in=270] (i1);
\draw[edge] (i)  to[out=60,in=300] (s1);
\draw[edge] (m)  to[out=90,in=270] (c1);
\draw[edge] (c)  to[out=60,in=300] (m1);
\draw[edge] (j)  to[out=90,in=270] (a1);
\draw[edge] (a)  to[out=60,in=300] (j1);
\end{tikzpicture}
\caption{The unique 1-factor $F_i$ (bold edges) of $G \setminus V_i$ in the case where $H$ is the $4$-cycle  in Example~\ref{Ex:graph2}.}
\label{uni1fac}
\end{center}
\end{figure}

We now turn our attention to digraphs of low degree requiring many derangements to generate them.

\begin{example}\label{Ex:digraph}
Let $D$ be the digraph with vertex set $V=\{w_i : i \in \Z\}$ and arc set
\[E=\{(w_i,w_{i+1}), (w_i,w_{i-1}), (w_{2i-1},w_{2i+1}) : i\in\mathbb{Z}\}\]
depicted in Figure \ref{infinitedigraph}. Note that $D$ is connected and has maximum in- and out-degree $3$. In fact, the graph $G$ given in Example~\ref{Ex:graph} is the bipartite double of $D$ if we take, for each $i \in \Z$, $v_{2i-1}=(w_{2i-1},1)$, $u_{2i-1}=(w_{2i+1},2)$, $v_{2i}=(w_{2i},2)$ and $u_{2i-2}=(w_{2i},1)$.

For any positive integer $n$, consider the subset $T=\{w_{2i}:\ 1\leq i\leq n\}$ of $V$. We have $|N^+_D(T)|=n+1$, $|T|=n$, and $\sum_{x\in N^+_D(T)}\deg^-_D(x)-\sum_{x\in T}\deg^+_D(x)=3(n+1)-2n=n+3$. Since $n$ can be chosen to be any positive integer, by Theorem~\ref{T:digraphkCharacterisation}(ii), $D$ cannot be generated by any finite set of derangements. Intuitively, this is because each derangement in a putative generating set for $D$ generates at most one of the distance 2 arcs $\{(w_{2i-1},w_{2i+1}): i \in \mathbb{Z}\}$. However, by Theorem~\ref{T:infinite1ExtendibleDigraph} or by simple inspection, it can be seen that $D$ can be generated by infinitely many derangements.

In addition, $D$ has a (non-induced) subdigraph that requires exactly $k$ derangements to generate it for each positive integer $k$. For any positive integer $k$ consider the subdigraph $D_k$ of $D$ obtained by taking the induced subdigraph of $D$ with vertex set $V_k=\{w_i: 1\leq i\leq 2k+1\}$ and removing the arcs $(w_2,w_3)$ and $(w_{2k-1},w_{2k})$. (Note that $D_k$ does not correspond to the subgraph $G_k$ of $G$ described in Example~\ref{Ex:graph}.) It can be seen that there is a unique set of $k$ derangements that generates $D_k$. For example, for $k=2$ this set is
$\{(w_1w_3w_2)(w_4w_5),(w_1w_2)(w_3w_5w_4)\}$, and for $k=3$ it is
$\{(w_1w_3w_2)(w_4w_5)(w_6w_7),(w_1w_2)(w_3w_5w_4)(w_6w_7),(w_1w_2)(w_3w_4)(w_5w_7w_6)\}$.
Furthermore, by applying Theorem~\ref{T:digraphkCharacterisation}(ii) with $T=\{v_{2i}:\ 1\leq i\leq k\}$, we see that $D_k$ cannot be generated by fewer than $k$ derangements.
\end{example}

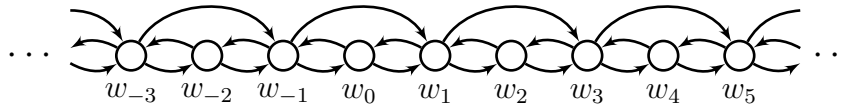
\begin{figure}[H]
\begin{center}
\begin{tikzpicture}[scale=1.0]
\tikzset{vertex/.style = {shape=circle,draw, line width=1pt,opacity=1.0}}
\tikzset{arc/.style = {->,> = latex', line width=1pt,opacity=1.0}}
\tikzset{edge/.style = {-,> = latex', line width=1pt,opacity=1.0}}

\node[vertex] (a) at (-4,0) {};
\node  at (-4,-0.5) {$w_{-3}$};
\node[vertex] (b) at (-3,0) {};
\node  at (-3,-0.5) {$w_{-2}$};
\node[vertex] (c) at  (-2,0) {};
\node  at (-2,-0.5) {$w_{-1}$};
\node[vertex] (d) at  (-1,0) {};
\node  at (-1,-0.5) {$w_{0}$};
\node[vertex] (e) at  (0,0) {};
\node  at (0,-0.5) {$w_{1}$};
\node[vertex] (f) at  (1,0) {};
\node  at (1,-0.5) {$w_{2}$};
\node[vertex] (g) at (2,0) {};
\node  at (2,-0.5) {$w_{3}$};
\node[vertex] (h) at (3,0) {};
\node  at (3,-0.5) {$w_{4}$};
\node[vertex] (i) at (4,0) {};
\node  at (4,-0.5) {$w_{5}$};

\node  at (-5.3,0) {\large{$\cdots$}};
\draw[arc] (-4.8,0.6)  to[out=0,in=120] (a);
\draw[arc] (-4.8,-0.1) to[out=-30,in=-150] (a);
\draw[arc] (a) to[out=150,in=30] (-4.8,0.1);
\draw[arc] (a)  to[out=-30,in=-150] (b);
\draw[arc] (b) to[out=150,in=30] (a);
\draw[arc] (a)  to[out=60,in=120] (c);
\draw[arc] (b)  to [out=-30,in=-150] (c);
\draw[arc] (c) to[out=150,in=30] (b);
\draw[arc] (c)  to [out=-30,in=-150] (d);
\draw[arc] (d) to[out=150,in=30] (c);
\draw[arc] (c)  to[out=60,in=120] (e);
\draw[arc] (d)  to [out=-30,in=-150] (e);
\draw[arc] (e) to[out=150,in=30] (d);
\draw[arc] (e)  to [out=-30,in=-150] (f);
\draw[arc] (f) to[out=150,in=30] (e);
\draw[arc] (e)  to[out=60,in=120] (g);
\draw[arc] (f)  to [out=-30,in=-150] (g);
\draw[arc] (g) to[out=150,in=30] (f);
\draw[arc] (g)  to [out=-30,in=-150] (h);
\draw[arc] (h) to[out=150,in=30] (g);
\draw[arc] (g)  to[out=60,in=120] (i);
\draw[arc] (h)  to [out=-30,in=-150] (i);
\draw[arc] (i) to[out=150,in=30] (h);
\draw[arc] (i) to[out=-30,in=-150] (4.8,-0.1);
\draw[arc] (4.8,0.1) to[out=150,in=30] (i);
\draw[edge] (i) to[out=60,in=-175] (4.8,0.6);
\node  at (5.3,0) {\large{$\cdots$}};
\end{tikzpicture}
\caption{Digraph in Example~\ref{Ex:digraph} that can be generated by infinitely many derangements but not by finitely many.}
\label{infinitedigraph}
\end{center}
\end{figure}

\medskip
\noindent\textbf{Acknowledgments.} The first author was supported by Australian Research Council grants DP150100506 and FT160100048.

\end{document}